\documentclass[12pt]{amsart}

\usepackage{amssymb}
\usepackage{bm}
\usepackage{graphicx}
\usepackage{psfrag}
\usepackage{color}
\usepackage{url}
\usepackage{algpseudocode}
\usepackage{fancyhdr}
\usepackage{xy}
\input xy
\xyoption{all}

\newtheorem*{maintheorem*}{Main Theorem}
\newtheorem{theorem}{Theorem}[section]
\newtheorem{prop}[theorem]{Proposition}
\newtheorem{conj}[theorem]{Conjecture}
\newtheorem{lemma}[theorem]{Lemma}
\newtheorem{cor}[theorem]{Corollary}

\theoremstyle{definition}
\newtheorem{definition}[theorem]{Definition}
\newtheorem{remark}[theorem]{Remark}
\newtheorem{example}[theorem]{Example}
\numberwithin{equation}{section}

\newcommand{\nn}{\mathbb{N}}
\newcommand{\qq}{\mathbb{Q}}
\newcommand{\rr}{\mathbb{R}}
\newcommand{\zz}{\mathbb{Z}}
\newcommand\pval{\mathsf{v}_p}

\keywords{Puiseux monoids, system of sets of lengths, sets of lengths, realization theorem, characterization problem, atomic monoids}

\begin{document}
	
	\mbox{}
	\title{Systems of sets of lengths of Puiseux monoids}
	\author{Felix Gotti}
	\address{Department of Mathematics\\UC Berkeley\\Berkeley, CA 94720}
	\email{felixgotti@berkeley.edu}
	\date{\today}
	
	\begin{abstract}
		In this paper we study the system of sets of lengths of non-finitely generated atomic Puiseux monoids (a Puiseux monoid is an additive submonoid of $\qq_{\ge 0}$). We begin by presenting a BF-monoid $M$ with full system of sets of lengths, which means that for each subset $S$ of $\zz_{\ge 2}$ there exists an element $x \in M$ whose set of lengths $\mathsf{L}(x)$ is $S$. It is well known that systems of sets of lengths do not characterize numerical monoids. Here, we prove that systems of sets of lengths do not characterize non-finitely generated atomic Puiseux monoids. In a recent paper, Geroldinger and Schmid found the intersection of systems of sets of lengths of numerical monoids. Motivated by this, we extend their result to the setting of atomic Puiseux monoids. Finally, we relate the sets of lengths of the Puiseux monoid $P = \langle 1/p \mid p \ \text{is prime} \rangle$ with the Goldbach's conjecture; in particular, we show that $\mathsf{L}(2)$ is precisely the set of Goldbach's numbers.
	\end{abstract}

\maketitle

\vspace{-6pt}

\section{Introduction} \label{sec:intro}

Factorization theory originated from algebraic number theory, in particular, from the fact that the ring of integers $\mathcal{O}_K$ of an algebraic number field $K$ fails in general to be factorial. An integral domain is called half-factorial if any two factorizations of the same element involve the same number of irreducibles. In the 1950's, L.~Carlitz \cite{lC60} proved that a ring of integers $\mathcal{O}_K$ is half-factorial if and only if its class group $\mathcal{C}(\mathcal{O}_K)$ has order at most $2$. He also proposed to study further connections between the phenomenon of non-unique factorizations of $\mathcal{O}_K$ and the structure of $\mathcal{C}(\mathcal{O}_K)$. In general, many algebraic invariants of Noetherian domains can also be used to understand to which extent such integral domains fail to be factorial.

If the set $\mathsf{Z}(x)$ of factorizations into irreducibles of a given element $x$ in an integral domain is not a singleton, many interesting questions naturally emerge to understand how bizarre $\mathsf{Z}(x)$ might be. For instance, what is the size of $\mathsf{Z}(x)$ or what is the subset $\mathsf{L}(x) \subset \nn$ consisting of the lengths of elements of $\mathsf{Z}(x)$? Or, perhaps, how similar or close are two given factorizations in $\mathsf{Z}(x)$? Many statistics and algebraic invariants have been introduced to answer these and other similar questions induced by the phenomenon of non-unique factorizations. The study of such algebraic invariants in integral domains (see \cite{dD96}) and, more recently, in the abstract setting of commutative cancellative monoids (see \cite{aG16} and references therein) have given shape to the modern factorization theory.

Perhaps the most investigated factorization invariant is the system of sets of lengths. If $M$ is a commutative cancellative monoid and $x \in M$ can be written as a product of $n$ irreducibles, then $n$ is called a length of $x$, and the set $\mathsf{L}(x)$ comprising all the lengths of $x$ is called the set of lengths of $x$. In addition, the set $\{\mathsf{L}(x) \mid x \in M\}$ is called the system of sets of lengths of $M$. The reader might have noticed that if $M$ is factorial, then $|\mathsf{L}(x)| = 1$ for all $x \in M$. Clearly, $|\mathsf{L}(x)| = 1$ for all $x \in M$ means that $M$ is half-factorial.

An extensive family of commutative cancellative monoids with a wild atomic structure and a rich factorization theory is hidden inside the set of nonnegative rational numbers. The members of this family, additive submonoids of $\qq_{\ge 0}$, are called Puiseux monoids. The atomic structure of Puiseux monoids has only been studied recently (see \cite{fG17a}, \cite{fG17b}, \cite{GG17}). In the present manuscript, we study the system of sets of lengths of non-finitely generated atomic Puiseux monoids.

We begin the next section by establishing the notation of commutative semigroup theory we shall be using later. Then we recall the concepts of factorization theory that are required to follow this paper. In particular, we formally define what a factorization is, and we introduce the factorization invariants that will play some role in the following sections. Finally, a brief insight into the atomic structure of Puiseux monoids is provided.

In Section~\ref{sec:BF monoid with full elasticity}, we begin our journey into the system of sets of lengths of Puiseux monoids. It was recently proved in \cite{GS17} that for each $S \subseteq \zz_{\ge 2}$ we can find a numerical monoid $N$ and $x \in N$ with $\mathsf{L}(x) = S$. We study the same question, but in the setting of non-finitely generated Puiseux monoids. As a result, we construct a BF-monoid $M$ whose system of sets of lengths is as large as we can expect, i.e., for each $S \subseteq \zz_{\ge 2}$ there exists $x \in M$ such that $\mathsf{L}(x) = S$. Note that in our setting the monoid does not depend on the choice of the set $S$.

Can we characterize the members of certain given family of atomic monoids by their systems of sets of lengths? This is an important question in factorization theory, which is known as the Characterization Problem. For example, it is conjectured that Krull monoids over finite abelian groups with prime divisors in all classes whose Davenport constant is at least $4$ can be characterized by their systems of sets of lengths. On the other hand, it was proved in \cite{ACHP07} that systems of sets of lengths do not characterize numerical monoids. In Section~\ref{sec:characterization problem}, we show that non-finitely generated Puiseux monoids cannot be characterized by their systems of sets of lengths.

In Section~\ref{sec:intersections of sets of lengths}, we construct an atomic Puiseux monoid $M$ that is completely non-half-factorial, meaning that each $x \in M \setminus \{0\}$ that is not irreducible satisfies $|\mathsf{L}(x)| \ge 2$. Then, motivated by \cite[Section~4]{GS17}, we study the intersection of systems of sets of lengths of atomic Puiseux monoids. The construction of a completely non-half-factorial Puiseux monoid will allow us to give a version of \cite[Theorem~4.1]{GS17} in the setting of atomic Puiseux monoids.

As most of the results in this paper are about cardinality restrictions and partial descriptions of sets of lengths rather than their explicit determination, we feel the need to argue how complex are explicit computations of sets of lengths, even for the simplest atomic Puiseux monoids. Thus, in Section~\ref{sec:Goldbach conjecture} we show that in the elementary Puiseux monoid $\langle 1/p \mid p \text{ is prime} \rangle$ the set of lengths $\mathsf{L}(2)$ is precisely the set of Goldbach's numbers. This, in particular, implies that an explicit description of $\mathsf{L}(2)$ is as hard as the Goldbach's conjecture.\\

\section{Background} \label{sec:background}

To begin with let us introduce the fundamental concepts related to our exposition as an excuse to establish the notation we need. The reader can consult Grillet \cite{pG01} for information on commutative semigroups and Geroldinger and Halter-Koch \cite{GH06b} for extensive background in non-unique factorization theory of commutative domains and monoids. These two references also provide definitions for most of the undefined terms we mention here.

Throughout this sequel, we let $\mathbb{N}$ denote the set of positive integers, and we set $\nn_0 := \nn \cup \{0\}$. For $X \subseteq \rr$ and $r \in \rr$, we set $X_{\le r} := \{x \in X \mid x \le r\}$; with a similar spirit we use the symbols $X_{\ge r}$, $X_{< r}$, and $X_{> r}$. If $q \in \qq_{> 0}$, then we call the unique $a,b \in \nn$ such that $q = a/b$ and $\gcd(a,b)=1$ the \emph{numerator} and \emph{denominator} of $q$ and denote them by $\mathsf{n}(q)$ and $\mathsf{d}(q)$, respectively.

The unadorned term \emph{monoid} always means commutative cancellative monoid. As most monoids here is assumed to be commutative, unless otherwise specified we will use additive notation. We let $M^\bullet$ denote the set $M \! \setminus \! \{0\}$. For $a,c \in M$, we say that $a$ \emph{divides} $c$ \emph{in} $M$ and write $a \mid_M c$ provided that $c = a + b$ for some $b \in M$. We write $M = \langle S \rangle$ when $M$ is generated by a set $S$. We say that $M$ is \emph{finitely generated} if it can be generated by a finite set; otherwise, $M$ is said to be \emph{non-finitely generated}.

A non-invertible element $a \in M$ is an \emph{atom} (or \emph{irreducible}) if for each pair of elements $u,v \in M$ such that $a = u + v$ either $u$ or $v$ is invertible. Let $\mathcal{A}(M)$ denote the set of atoms of $M$. Every monoid $M$ in this paper will be \emph{reduced}, which means that $0$ is the only invertible element of $M$. This clearly implies that $\mathcal{A}(M)$ will be contained in each generating set of $M$. If $\mathcal{A}(M)$ generates $M$, then $M$ is said to be \emph{atomic}. Monoids addressed in this article are all atomic.
We say that a multiplicative monoid $F$ is \emph{free abelian} on $P \subset F$ if every element $a \in F$ can be written uniquely in the form
\[
	a = \prod_{p \in P} p^{\pval(a)},
\]
where $\pval(a) \in \nn_0$ and $\pval(a) > 0$ only for finitely many elements $p \in P$. It is well known that for each set $P$, there exists a unique (up to canonical isomorphism) monoid $F$ such that $F$ is free abelian on $P$. When we want to emphasize the relation between $P$ and $F$, we denote $F$ by $\mathcal{F}(P)$. It follows by the fundamental theorem of arithmetic that the multiplicative monoid $\nn$ is free abelian on the set of prime numbers. In this case, we can extend $\pval$ to $\qq_{\ge 0}$ as follows. For $r \in \qq_{> 0}$ let $\pval(r) := \pval(\mathsf{n}(r)) - \pval(\mathsf{d}(r))$ and set $\pval(0) = \infty$. The map $\pval \colon \qq_{\ge 0} \to \zz$, called the $p$-\emph{adic valuation} on $\qq_{\ge 0}$, satisfies the following two properties:
\begin{eqnarray}
	&& \pval(rs) = \pval(r) + \pval(s) \ \ \text{for all} \ \ r,s \in \qq_{\ge 0}; \vspace{10pt} \\
	&& \pval(r + s) \ge \min \{ \pval(r), \pval(s)\} \ \ \text{for all} \ \ r,s \in \qq_{\ge 0}. \vspace{3pt}
\end{eqnarray}

The free abelian monoid on $\mathcal{A}(M)$, denoted by $\mathsf{Z}(M)$, is called the \emph{factorization monoid} of $M$, and the elements of $\mathsf{Z}(M)$ are called \emph{factorizations}. If $z = a_1 \dots a_n$ is a factorization in $\mathsf{Z}(M)$ for some $n \in \nn_0$ and $a_1, \dots, a_n \in \mathcal{A}(M)$, then $n$ is called the \emph{length} of $z$ and is denoted by $|z|$. The unique homomorphism
\[
	\phi \colon \mathsf{Z}(M) \to M \ \ \text{satisfying} \ \ \phi(a) = a \ \ \text{for all} \ \ a \in \mathcal{A}(M)
\]
is called the \emph{factorization homomorphism} of $M$, and for each $x \in M$ the set
\[
	\mathsf{Z}(x) := \phi^{-1}(x) \subseteq \mathsf{Z}(M)
\]
is called the \emph{set of factorizations} of $x$. By definition, we set $\mathsf{Z}(0) = \{0\}$. Note that the monoid $M$ is atomic if and only if $\mathsf{Z}(x)$ is nonempty for all $x \in M$. For each $x \in M$, the \emph{set of lengths} of $x$ is defined by
\[
	\mathsf{L}(x) := \{|z| \ | \ z \in \mathsf{Z}(x)\}.
\]
A monoid $M$ is \emph{half-factorial} if $|\mathsf{L}(x)| = 1$ for all $x \in M$. On the other hand, we say that the monoid $M$ is \emph{completely non-half-factorial} if $|\mathsf{L}(x)| = \infty$ for all $x \in M^\bullet \setminus \mathcal{A}(M)$. If $\mathsf{L}(x)$ is a finite set for all $x \in M$, then we say that $M$ is a BF-\emph{monoid}. The \emph{system of sets of lengths} of $M$ is defined by
\[
	\mathcal{L}(M) := \{\mathsf{L}(x) \mid x \in M\}.
\]
In \cite{aG16} the interested reader can find a friendly introduction to sets of lengths and the role they play in factorization theory. In general, sets of lengths and systems of sets of lengths are factorization invariants of atomic monoids that have received significant attention in recent years (see, for instance, \cite{ACHP07,FG08,GS16}). 

A very special family of atomic monoids is that one comprising all \emph{numerical monoids}, cofinite submonoids of $\nn_0$. Each numerical monoid $N$ has a unique minimal set of generators, which is finite; such a unique minimal generating set is precisely $\mathcal{A}(N)$. As a result, every numerical monoid is atomic and contains only finitely many atoms. An introduction to the realm of numerical monoids can be found in~\cite{GR09}.

Recall that an additive submonoid of $\qq_{\ge 0}$ is called a \emph{Puiseux monoid}. Puiseux monoids are a natural generalization of numerical monoids. However, in general, the atomic structure of Puiseux monoids differs significantly from that one of numerical monoids. Puiseux monoids are not always atomic; for instance, consider $\langle 1/2^n \mid n \in \nn\rangle$. On the other hand, if an atomic Puiseux monoid $M$ is not isomorphic to a numerical monoid, then $\mathcal{A}(M)$ is infinite. It is also useful to know that if a Puiseux monoid does not contain $0$ as a limit point, then it is atomic. Indeed, the following stronger statement, which is a direct consequence of \cite[Theorem~4.5]{fG16}, holds.

\begin{theorem} \label{theo:sufficient condition for atomicity}
	Let $M$ be a Puiseux monoid. If $0$ is not a limit point of $M$, then $M$ is a \emph{BF}-monoid.
\end{theorem}

The atomic structure and factorization theory of Puiseux monoids has only been studied recently (see \cite{fG17b} and references therein). \\

\section{A BF-Puiseux monoid with full system of sets of lengths} \label{sec:BF monoid with full elasticity}

Given a factorization invariant $\mathsf{f}$ of atomic monoids (resp., of elements of atomic monoids) and certain class of atomic monoids $\mathcal{C}$, the question of whether there exists $M \in \mathcal{C}$ (resp., $M \in \mathcal{C}$ and $x \in M$) such that $\mathsf{f}(M)$ (resp., $\mathsf{f}_M(x)$) equals some prescribed value is called a \emph{realization problem}. Besides the sets of lengths, there are many factorization invariants, including the set of distances and the catenary degree (definitions can be found in~\cite{GH06b}), for which the realization problem restricted to several classes of atomic monoids have been studied lately. Indeed, theorems in this direction have been established in \cite{CGP14,GS17a,OP17,wS09}.

In this section, we study a realization problem when the factorization invariant $\mathsf{f}$ is the system of sets of lengths and the class $\mathcal{C}$ is that one comprising all Puiseux monoids which are also BF-monoids. We take our prescribed value to be the collection of sets
\[
	\mathcal{S} = \big\{\{0\}, \{1\}, S \mid S \subseteq \zz_{\ge 2} \ \text{and} \ |S| < \infty \big\}.
\]
As the only nonzero elements of an atomic monoid $M$ having factorizations of length~$1$ are the atoms, it follows that $\mathcal{L}(M) \subseteq \{\{0\}, \{1\}, S \mid S \subseteq \zz_{\ge 2}\}$. Therefore, when $M$ is a BF-monoid we obtain that $\mathcal{L}(M) \subseteq \mathcal{S}$.

\begin{definition}
	We say that a BF-monoid $M$ has \emph{full system of sets of lengths} provided that $\mathcal{L}(M) = \mathcal{S}$.
\end{definition}

We positively answer our realization question by constructing (in the proof of Theorem~\ref{thm:PM with full system of sets of lengths}) a Puiseux monoid with full system of sets of lengths. Note that families of monoids and domains having full systems of sets of lengths have been found and studied before. It was proved by Kainrath \cite{fK99} that Krull monoids having infinite class groups with primes in each class have full systems of sets of lengths. On the other hand, Frish \cite{sF13} proved that the subdomain $\text{Int}(\zz)$ of $\zz[x]$ also has full system of sets of lenghts; this result has been recently generalized \cite{FNR17}, as we show in Example~\ref{ex:integer-valued polynomials have FSSL}.

\begin{example} \label{ex:Krull monoids with cyclic class groups and a prime in each class have FSSL}
	Let $M$ be Krull monoid, and let $G$ be the class group of $M$. Suppose that $G$ is infinite and that every class contains at least a prime. Therefore for each nonempty finite subset $L$ of $\zz_{\ge 2}$ and every function $f \colon L \to \nn$, there exists $x \in M$ satisfying the following two conditions:
	\begin{enumerate}
		\item $\mathsf{L}(x) = L$, and \vspace{1pt}
		\item $|\{z \in \mathsf{Z}(x) \ | \ |z| = k\}| \ge f(k)$ for each $k \in L$.
	\end{enumerate}
	In particular, $M$ has full system of sets of lengths. In this example we illustrate a simplified version of \cite[Theorem~1]{fK99}. We refer the reader to \cite[Section~3]{aG16} not only for the definition of Krull monoids but also for a variety of examples showing up in diverse areas of mathematics.
\end{example}

\begin{example} \cite[Theorem~4.1]{FNR17} \label{ex:integer-valued polynomials have FSSL}
	Let $\mathcal{O}_K$ be the ring of integers of a given number field $K$. In addition, take $m_1, \dots, m_n \in \nn$ such that $m_1 < \dots <  m_n$. Let $\text{Int}(\mathcal{O}_K)$ denote the subring of integer-valued polynomials of $K[x]$ (i.e., the subring of polynomials of $K[x]$ stabilizing $\mathcal{O}_K$). Then there exists $p(x) \in \text{Int}(\mathcal{O}_K)$ and $z_1, \dots, z_n \in \mathsf{Z}(p(x))$ satisfying that $|z_i| = m_i + 1$ for each $i = 1, \dots, n$. As a result, the domain $\text{Int}(\mathcal{O}_K)$ has full system of sets of lengths.
\end{example}

The following result, which is a crucial tool in our construction, is a simplified version of a recent realization theorem by Geroldinger and Schmid.

\begin{theorem} \cite[Theorem~3.3]{GS17} \label{thm:simplified version of set of length realizability for NS}
	For every nonempty finite subset $S$ of $\zz_{\ge 2}$, there exists a numerical monoid $N$ and $x \in N$ such that $\mathsf{L}(x) = S$.
\end{theorem}

Theorem~\ref{thm:simplified version of set of length realizability for NS} implies, in particular, that every nonempty finite subset $S$ of $\zz_{\ge 2}$ can be realized as the set of lengths of an element inside certain Puiseux monoid. In principle, the choices of both the Puiseux monoid and the element depend on the set $S$. However, the existence of a Puiseux monoid with full system of sets of lengths will eliminate the former of these two dependences. We will create such a Puiseux monoid by "gluing" together a countable family of numerical monoids, each of them containing an element whose set of lengths is a specified finite subset of $\zz_{\ge 2}$. Clearly, we should glue the numerical monoids carefully enough so that none of the specified sets of lengths is lost in the process. First, let us state the next lemma.

\begin{lemma} \label{lem:scaling numerical semigroups preserves lengths}
	If $N$ is a submonoid of $(\nn_0, +)$ and $q \in \qq_{> 0}$, then $qN$ is a finitely generated (Puiseux) monoid satisfying that $\mathsf{L}_{N}(x) = \mathsf{L}_{qN}(qx)$ for every $x \in N$.
\end{lemma}

\begin{proof}
	It suffices to notice that, for every $q \in \qq_{> 0}$, multiplication by $q$ yields an isomorphism from $N$ to $qN$.
\end{proof}

\begin{theorem} \label{thm:PM with full system of sets of lengths}
	There is an atomic Puiseux monoid with full system of sets of lengths.
\end{theorem}

\begin{proof}
	Because the collection of all finite subsets of $\zz_{\ge 2}$ is countable, we can list them in a sequence, say $\{S_n\}$. Now we recursively construct a sequence of finitely generated Puiseux monoids $\{M_n\}$, a sequence of rational numbers $\{x_n\}$, and a sequence of odd prime numbers $\{p_n\}$ satisfying the following conditions:
	\begin{enumerate}
		\item $x_n \in M_n$ and $\mathsf{L}_{M_n}(x_n) = S_n$; \vspace{2pt}
		\item the set $A_n$ minimally generating $M_n$ satisfies that $\mathsf{d}(a) = p_n$ for every $a \in A_n$; \vspace{2pt}
		\item $p_n > \max \big\{2x_n, 2a \mid a \in A_n \big\}$; \vspace{2pt}
		\item $\max A_n < \min A_{n+1}$ for every $n \in \nn$.
	\end{enumerate}
	To do this, use Theorem~\ref{thm:simplified version of set of length realizability for NS} to find a numerical monoid $M'_1$ minimally generated by $A'_1$ such that $\mathsf{L}(x'_1) = S_1$ for some $x'_1 \in M'_1$. Then take $p_1$ to be a prime number satisfying that $p_1 > \max\{2x'_1, 2a' \mid a' \in A'_1\}$. Now define
	\[
		M_1 = \frac{p_1 - 1}{p_1}M'_1 \quad \text{and} \quad x_1 = \frac{p_1 - 1}{p_1}x'_1.
	\]
	By Lemma~\ref{lem:scaling numerical semigroups preserves lengths}, the element $x_1$ satisfies condition (1). Conditions (2), (3), and (4) follow immediately. Suppose now that we have already constructed a set of Puiseux monoids $\{M_1, \dots, M_n\}$, a set of rational numbers $\{x_1, \dots, x_n\}$, and a set of prime numbers $\{p_1, \dots, p_n\}$ such that the above conditions are satisfied. By Theorem~\ref{thm:simplified version of set of length realizability for NS}, there exists a submonoid $M'_{n+1}$ of $(\nn_0,+)$ minimally generated by $A'_{n+1}$ which contains an element $x'_{n+1}$ with $\mathsf{L}_{M'_{n+1}}(x'_{n+1}) = S_{n+1}$. By Lemma~\ref{lem:scaling numerical semigroups preserves lengths}, we can assume that the elements of $A'_{n+1}$ are large enough that $\max A_n < \min A'_{n+1}$. Now choose a prime number $p_{n+1}$ sufficiently large such that $p_{n+1} \nmid a$ for any $a' \in A'_{n+1}$,
	\begin{align} \label{eq:condition (3) for p_{n+1}}
		p_{n+1} > \max \bigg\{2 \, \frac{p_{n+1} - 1}{p_{n+1}} x'_{n+1}, \, 2 \, \frac{p_{n+1} - 1}{p_{n+1}} \, a' \ \bigg{|} \ a' \in A'_{n+1} \bigg\},
	\end{align}
	and
	\begin{align} \label{eq:condition (4) for p_{n+1}}
		\max A_n < \frac{p_{n+1} - 1}{p_{n+1}} \min A'_{n+1}.
	\end{align}
	Finally, set
	\[
		M_{n+1} = \frac{p_{n+1} - 1}{p_{n+1}} M'_{n+1} \quad \text{and} \quad x_{n+1} = \frac{p_{n+1} - 1}{p_{n+1}}x'_{n+1}.
	\]
	By Lemma~\ref{lem:scaling numerical semigroups preserves lengths}, it follows that $\mathsf{L}_{M_{n+1}}(x_{n+1}) = \mathsf{L}_{M'_{n+1}}(x'_{n+1}) = S_{n+1}$, which is condition (1). The fact that $p_{n+1} \nmid a$ for any $a \in A'_{n+1}$ yields condition (2). Finally, conditions (3) and (4) follows from (\ref{eq:condition (3) for p_{n+1}}) and (\ref{eq:condition (4) for p_{n+1}}), respectively. Hence our sequence of finitely generated Puiseux monoids $\{M_n\}$ satisfies the desired conditions.
	
	Now consider the Puiseux monoid $M = \langle A \rangle$, where $A := \cup_{n \in \nn} A_n$. In addition, let $\{k_n\}$ be a sequence of positive integers such that
	\[
		A_n =: \bigg\{\frac{p_n - 1}{p_n} a_{ni} \ \bigg{|} \ 1 \le i \le k_n \bigg\},
	\]
	where $a_{ni} \in \nn$ for every $n$ and $i = 1, \dots, k_n$. We verify now that $M$ is atomic and $\mathcal{A}(M) = A$. To do so, suppose that for some $m \in \nn$ and $j \in \{1,\dots, k_m\}$ we can write
	\begin{align} \label{eq:equation for atomicity}
		\frac{p_m - 1}{p_m}a_{mj} = \sum_{n=1}^s \sum_{i=1}^{k_n} c_{ni} \frac{p_n - 1}{p_n} a_{ni},
	\end{align}
	where $s \in \nn$ and $c_{ni} \in \nn_0$ for every $n=1,\dots,s$ and $i=1,\dots, k_n$. If $m > s$, then the $p_m$-adic valuation of the right-hand side of (\ref{eq:equation for atomicity}) would be nonnegative, contradicting that $p_m \nmid (p_m - 1) a_{mj}$. Therefore assume that $m \le s$. Now set $q_s = p_1 \dots p_s$. After multiplying \eqref{eq:equation for atomicity} by $q_s$ and taking modulo $p_m$, we find that
	\begin{align} \label{eq:equation for atomicity 1}
		p_m \ \bigg{|} \ \frac {p_m - 1}{p_m} q_s \, a_{mj} - \frac {p_m - 1}{p_m} q_s \sum_{i=1}^{k_m} c_{mi} a_{mi}.
	\end{align}
	Since $\gcd \big(p_m, \frac{p_m - 1}{p_m} q_s\big) = 1$, there exists $N \in \nn_0$ such that
	\[
		a_{mj}  = Np_m + \sum_{i=1}^{k_m} c_{mi} a_{mi}.
	\]
	The fact that $p_m > \max 2A_m \ge a_m$ (by condition (3)) now implies that $N=0$ and, therefore, one obtains that $a_{mj} = c_{m1} a_{m1} + \dots + c_{mk_m} a_{mk_m}$. As $A_m$ generates $M_m$ minimally, it follows that $c_{mj} = 1$ and $c_{mi} = 0$ for each $i \neq j$. This, along with (\ref{eq:equation for atomicity}), implies that $c_{ni} = 0$ for every $(n,i) \neq (m,j)$. As a result, $\mathcal{A}(M) = A$.
	
	Now we show that every nonempty finite subset of $\zz_{\ge 2}$ is a set of lengths of $M$. This amounts to verifying that $\mathsf{L}_M(x_\ell) = \mathsf{L}_{M_\ell}(x_\ell)$ for every $\ell \in \nn$. Recall that
	\begin{align*} 
		x_\ell = \frac{p_\ell - 1}{p_\ell} x'_\ell, \ \text{ where } \ \mathsf{L}_{M'_\ell}(x'_\ell) = S_\ell.
	\end{align*}
	Suppose that for some $\ell, t \in \nn$ with $\ell \le t$ we have
	\begin{align} \label{eq:length}
		\frac{p_\ell - 1}{p_\ell} x'_\ell = \sum_{n=1}^t \sum_{i=1}^{k_n} c_{ni} \frac{p_n - 1}{p_n} a_{ni} = \sum_{n \in [t] \setminus \{\ell\}} \frac{p_n - 1}{p_n} \sum_{i=1}^{k_n} c_{ni} a_{ni} + \frac {p_\ell - 1}{p_\ell} \sum_{i=1}^{k_{\ell}} c_{\ell i} a_{\ell i},
	\end{align}
	where $t \in \nn$ and $c_{ni} \in \nn_0$ for every $n=1,\dots,t$ and $i = 1, \dots, k_n$. Multiplying the equality (\ref{eq:length}) by $q_t = p_1 \dots p_t$ and taking modulo $p_\ell$, one can see that
	\begin{align}
		p_\ell \ \bigg{|} \ \frac{p_\ell - 1}{p_\ell} q_t \, x'_\ell - \frac{p_\ell - 1}{p_\ell} q_t \, \sum_{i=1}^{k_\ell} c_{\ell i} a_{\ell i}.
	\end{align}
	Once again, the fact that $\gcd(p_\ell, \frac{p_\ell - 1}{p_\ell})$ implies the existence of $N' \in \nn_0$ such that
	\[
		x'_\ell = N'p_\ell + \sum_{i=1}^{k_\ell} c_{\ell i} a_{\ell i}.
	\]
	However, as $p_\ell > 2x_\ell \ge x'_\ell$ (by condition (3)), we have that $N' = 0$ and, as a consequence, $c_{\ell 1} + \dots + c_{\ell k_\ell} \in \mathsf{L}_{M'_\ell}(x'_\ell)$. The fact that $x'_\ell = c_{\ell 1} a_{\ell 1} + \dots + c_{\ell k_\ell} a_{\ell k_\ell}$, along with equality (\ref{eq:length}), immediately implies that $c_{ni} = 0$ for every $n \neq \ell$ and $i = 1,\dots,k_n$. As a result,
	\[
		\sum_{n=1}^t \sum_{i=1}^{k_n} c_{ni} = \sum_{i=1}^{k_{\ell}} c_{\ell i} \in \mathsf{L}_{M'_\ell}(x'_\ell) = \mathsf{L}_{M_\ell}(x_\ell).
	\]
	Then the inclusion $\mathsf{L}_{M}(x_\ell) \subseteq \mathsf{L}_{M_\ell}(x_\ell)$ holds. As the reverse inclusion clearly holds, we conclude that $\mathsf{L}_M(x_\ell) = \mathsf{L}_{M_\ell}(x_\ell)$ for every $\ell \in \nn$. Thus, each subset of $\zz_{\ge 2}$ can be realized as the set of lengths of some element in $M$. Because $\{0\}$ and $\{1\}$ can be obviously realized, we get that
	\begin{align} \label{eq:inclusion of systems of sets of lengths}
		\mathcal{L}(M) \supseteq \big\{\{0\}, \{1\}, S \subset \zz_{\ge 2} \ | \ |S| < \infty \big\}.
	\end{align}
	Finally, it is easily seen that condition (4) ensures that $M$ does not contain $0$ as a limit point. So it follows by Theorem~\ref{theo:sufficient condition for atomicity} that $M$ is a BF-monoid. Hence every set of lengths of $M$ is finite, which yields the reverse inclusion of (\ref{eq:inclusion of systems of sets of lengths}).
\end{proof}

\begin{remark}
	Theorem~\ref{thm:PM with full system of sets of lengths} does not follow from Example~\ref{ex:Krull monoids with cyclic class groups and a prime in each class have FSSL} via transfer homomorphisms; indeed, it was proved in \cite{fG17b} that the only nontrivial Puiseux monoids that are transfer Krull are those isomorphic to $(\nn_0,+)$. We refer the reader to \cite[Section~4]{aG16} for the definition and applications of transfer homomorphisms. \\
\end{remark}

\section{A Few Words on the Characterization Problem} \label{sec:characterization problem}

The question of whether the arithmetical information encoded in the phenomenon of nun-unique factorization of the ring of integers $\mathcal{O}_K$ of a given algebraic number field $K$ suffices to characterize the class group $\mathcal{C}(\mathcal{O}_K)$ dated back to the mid-nineteenth century. In the 1970's, Narkiewicz proposed the more general question of whether the arithmetic describing the non-uniqueness of factorizations in a Krull domain could be used to characterize its class group. For affirmative answer to this, the reader might want to consult \cite[Sections~7.1 and 7.2]{GH06b}. The next conjecture, also known as the Characterization Problem, is still open. However, for an overview of results where the statement of the conjecture holds under certain extra conditions, we refer the reader to \cite[Theorem~23]{aG16}.

\begin{conj}
	Let $M$ and $M'$ be Krull monoids with respective finite abelian class groups $G$ and $G'$ each of their classes contains at least one prime divisor. Assume also that $\mathsf{D}(G) \ge 4$. If $\mathcal{L}(M) = \mathcal{L}(M')$, then $M \cong M'$.
\end{conj}

Because the system of sets of lengths encodes significant information about the arithmetic of factorizations of an atomic monoid, further questions in the same spirit of the above conjecture naturally arise. For instance, we might wonder whether, in a specified family $\mathcal{F}$ of atomic monoids, a member is determined up to isomorphisms by its system of sets of lengths. It was proved in \cite{ACHP07} that the answer is negative when $\mathcal{F}$ is the family of all numerical monoids. In this section, we use Theorem~\ref{thm:PM with full system of sets of lengths} to answer the same question when $\mathcal{F}$ is taken to be the family of all non-finitely generated atomic Puiseux monoids.

\begin{lemma}\cite[Proposition~3.1]{fG17b} \label{lem:isomorphism between PM}
	The homomorphisms of Puiseux monoids are precisely those given by rational multiplication.
\end{lemma}

\begin{lemma} \label{lem:non-isomorphic PM}
	Let $P$ and $Q$ be disjoint infinite sets of primes, and let $M_P = \langle a_p \mid p \in P \rangle$ and $M_Q = \langle b_q \mid q \in Q \rangle$ be Puiseux monoids such that for all $p \in P$ and $q \in Q$, the denominators $\mathsf{d}(a_p)$ and $\mathsf{d}(b_q)$ are nontrivial powers of $p$ and $q$, respectively. Then $M_P \ncong M_Q$.
\end{lemma}

\begin{proof}
	Suppose, by way of contradiction, that $M_P \cong M_Q$. By Lemma~\ref{lem:isomorphism between PM}, there exists $r \in \qq_{> 0}$ such that $M_P = r M_Q$. If $q$ is a prime number in $Q$ such that $q \nmid \mathsf{n}(r)$, then $r b_q$ would be an element of $M_P$ such that $\mathsf{d}(r b_q)$ is divisible by a nontrivial power of $q$ and, therefore, $q \in P$. But this contradicts the fact that $P \cap Q$ is empty.
\end{proof}

A Puiseux monoid $M$ is \emph{bounded} if it can be generated by a bounded set of rational numbers; otherwise, $M$ is said to be \emph{unbounded}.

\begin{lemma} \cite[Lemma~3.4]{GG17} \label{lem:finite denominator set iff finitely-generated}
	Let $M$ be a nontrivial Puiseux monoid. Then $\mathsf{d}(M^\bullet)$ is bounded if and only if $M$ is finitely generated.
\end{lemma}

\begin{theorem}
	There exist two non-isomorphic non-finitely generated atomic Puiseux monoids with the same system of sets of lengths.
\end{theorem}

\begin{proof}
	Consider two infinite sets $P$ and $Q$ consisting of prime numbers such that $P \cap Q$ is empty. Now let us construct, as in the proof of Theorem~\ref{thm:PM with full system of sets of lengths}, a Puiseux monoids $M_P$ using only prime numbers in $P$ such that $M_P$ has full system of sets of lengths. The way we constructed the Puiseux monoid $M_P$ ensures that $\mathsf{d}(M_P^\bullet)$ is unbounded. So Lemma~\ref{lem:finite denominator set iff finitely-generated} implies that $M_P$ is non-finitely generated. Similarly, we can construct a non-finitely generated Puiseux monoid $M_Q$ with full system of sets of lengths by using only prime numbers in $Q$. As $P$ and $Q$ are disjoint, Lemma~\ref{lem:non-isomorphic PM} guarantees that $M_P$ and $M_Q$ are non-isomorphic. The fact that $M_P$ and $M_Q$ both have full systems of sets of lengths completes the proof. \\
\end{proof}

\section{Intersections of Systems of Sets of Lengths} \label{sec:intersections of sets of lengths}

In their recent paper \cite{GS17}, Geroldinger and Schmid studied the intersections of systems of sets of lengths of numerical monoids. In particular, they proved the following result.

\begin{theorem} \cite[Theorem~4.1]{GS17} \label{thm:intersecion of SSL for NM}
	We have
	\[
		\bigcap \mathcal{L}(H) = \big\{ \{0\}, \{1\}, \{2\} \big\},
	\]
	where the intersection is taken over all numerical monoids $H \subset \nn_0$. More precisely, for every $s \in \zz_{\ge 6}$, we have
	\[
		\bigcap_{|\mathcal{A}(H)|=s} \mathcal{L}(H) = \big\{ \{0\}, \{1\}, \{2\} \big\},
	\]
	and, for every $s \in \{2, 3, 4, 5\}$, we have
	\[
		\bigcap_{|\mathcal{A}(H)|=s} \mathcal{L}(H) = \big\{ \{0\}, \{1\}, \{2\}, \{3\} \big\},
	\]
	where the intersections are taken over all numerical monoids $H$ with the given properties.
\end{theorem}

In this section, we study the intersection of the systems of sets of lengths of atomic Puiseux monoids. We offer two versions of Theorem~\ref{thm:intersecion of SSL for NM}, namely Corollary~\ref{cor:interseciton of non-dense PM} and Corollary~\ref{cor:intersection of nontrivial atomic PM}. We will also construct a completely non-half-factorial Puiseux monoid.

\begin{prop} \label{prop:{2} is not in the SSL of a non-dense PM}
	Let $M$ be a Puiseux monoid such that $0$ is not a limit point of $M^\bullet$. Then $\{2\} \in \mathcal{L}(M)$.
\end{prop}

\begin{proof}
	Let $q = \inf M^\bullet$. As $0$ is not a limit point of $M$, it follows that $M$ is atomic (by \cite[Theorem~3.10]{fG17a}) and $q \neq 0$. If $q \in M$, then $q$ must be an atom. In this case, the minimality of $q$ ensures that $\{2\} = \mathsf{L}(2q) \in \mathcal{L}(M)$. Suppose, otherwise, that $q \notin M$. In this case, there must be an atom $a$ such that $q < a < 3q/2$. As the sum of any three atoms is greater than $3q$, the element $2a$ only contains factorizations of lengths~$2$. Hence $\{2\} = \mathsf{L}(2a) \in \mathcal{L}(M)$.
\end{proof}

Because the family of Puiseux monoids strictly contains the family of numerical monoids, Theorem~\ref{thm:intersecion of SSL for NM} guarantees that the intersection of all systems of sets of lengths of nontrivial atomic Puiseux monoids is contained in $\{ \{0\}, \{1\}, \{2\} \}$, the following corollary is an immediate consequence of Proposition~\ref{prop:{2} is not in the SSL of a non-dense PM}.

\begin{cor} \label{cor:interseciton of non-dense PM}
	We have
	\[
		\bigcap \mathcal{L}(M) = \big\{\{0\}, \{1\}, \{2\} \big\},
	\]
	where the intersection is taking over all nontrivial atomic Puiseux monoids $M$ not having $0$ as a limit point.
\end{cor}

In contrast to Proposition~\ref{prop:{2} is not in the SSL of a non-dense PM}, we will construct an atomic Puiseux monoid that does not contain the singleton $\{2\}$ as a set of lengths.

\begin{lemma} \label{lem:union of minimally generated PMs}
	Let $\{M_n\}$ be a sequence of atomic Puiseux monoids satisfying that $\mathcal{A}(M_n) \subset \mathcal{A}(M_{n+1})$ for each $n \in \nn$. Then $M = \cup_{n \in \nn} M_n$ is an atomic Puiseux monoid and
	\[
		\mathcal{A}(M) = \bigcup_{n \in \nn} \mathcal{A}(M_n).
	\]
\end{lemma}

\begin{proof}
	Because $M_n$ is atomic for each $n \in \nn$, the inclusion $\mathcal{A}(M_n) \subset \mathcal{A}(M_{n+1})$ implies that $M_n \subset M_{n+1}$. As a consequence, $M$ is a Puiseux monoid. Let $A = \cup_{n \in \nn} \mathcal{A}(M_n)$. It is clear that $A$ generates $M$. Therefore $\mathcal{A}(M) \subset A$. On the other hand, take $a \in \mathcal{A}(M_n)$ for some $n \in \nn$. Since $\mathcal{A}(M) \subset A$, it follows that $\mathsf{Z}_M(a) \subseteq \mathsf{Z}_{M_\ell}(a)$ for some $\ell > n$. Now the fact that $a \in \mathcal{A}(M_n) \subset \mathcal{A}(M_\ell)$ guarantees that $a$ has only one factorization in $M$, which has length~$1$. Thus, $a \in \mathcal{A}(M)$. Because $\mathcal{A}(M) = A$, we finally conclude that $M$ is atomic.
\end{proof}

For $S \subseteq \qq_{> 0}$, we define
\[
	\mathsf{d}_p(S) = \{\text{prime } p \mid p \text{ divides } \mathsf{d}(s) \text{ for some } s \in S \}.
\]

\begin{theorem} \label{thm:PM with no set of lengths equal {2}}
	There exists an atomic Puiseux monoid $M$ such that $\{2\} \notin \mathcal{L}(M)$.
\end{theorem}

\begin{proof}
	We will construct inductively a sequence of positive rational numbers $\{a_n\}$ and an increasing sequence of natural numbers $\{k_n\}$ so that each set $A_n = \{a_1, a_2, \dots, a_{k_n}\}$ minimally generates the Puiseux monoid $M_n = \langle A_n \rangle$. Take $(a_1, k_1) = (1,1)$ and $(a_2, k_2) = (2/3, 2)$. Clearly, $\mathcal{A}(M_1) = \{1\} = A_1$ and $\mathcal{A}(M_2) = \{1,2/3\} = A_2$. Now suppose that we have already found $k_1, \dots, k_n$ and $a_1, \dots, a_{k_n}$ satisfying the desired conditions. Let
	\[
		\bigg\{(s_{i1}, s_{i2}) \ \bigg{|} \ 1 \le i \le \binom{k_n + 1}{2} \bigg\}
	\]
	be an enumeration of the set $\{(a,b) \in [1, k_n] \times [1, k_n] \mid a \le b \}$. Now let us choose $\binom{k_n + 1}{2}$ odd prime numbers $p_1, \dots, p_{\binom{k_n + 1}{2}}$ such that $p_i \notin \mathsf{d}_p(M_n)$ and $p_i \nmid \mathsf{n}(a_{s_{i1}} + a_{s_{i2}})$ for any $i \in \{1, \dots, \binom{k_n + 1}{2}\}$. Finally, define
	\[
		a_{k_n + i} = \frac{a_{s_{i1}} + a_{s_{i2}}}{p_i}
	\]
	for every $i \in \big\{1, \dots, \binom{k_n + 1}{2} \big\}$, and take $k_{n+1} = k_n + \binom{k_n + 1}{2}$.
	
	We verify now that $A_{n+1} := \{a_1, a_2, \dots, a_{k_{n+1}}\}$ minimally generates the Puiseux monoid $M_{n+1} := \langle A_{n+1}\rangle$. Suppose, by contradiction, that for some $j \in \{1, \dots, k_{n+1}\}$,
	\begin{align} \label{eq:1}
		a_j = \sum_{i=1}^{k_{n+1}} c_i a_i,
	\end{align}
	where $c_1, \dots, c_{k_{n+1}} \in \nn_0$ and $c_j = 0$. Notice that if $j > k_n$, then we would obtain that the $p_t$-valuation (for $t = j - k_n$) of the right-hand side of (\ref{eq:1}) is nonnegative while the fact that $p \nmid \mathsf{n}(a_{s_{t1}} + a_{s_{t2}})$ implies that the $p_t$-valuation of the left-hand side is negative, a contradiction. Thus, assume that $j \le k_n$. Because the set
	\[
		\mathsf{d}_p(M_n) \cap \bigg\{p_i \ \bigg{|} \ 1 \le i \le \binom{k_n + 1}{2} \bigg\}
	\]
	is empty, for $i \in \big\{1, \dots, \binom{k_n+1}{2} \big\}$ the prime number $p_i$ divides the denominator of only one of the $a_m$'s on the right-hand side of (\ref{eq:1}), namely $a_{k_n + i}$. Therefore, after applying the $p_i$-adic valuation map to both sides of (\ref{eq:1}), we find that $p_i \mid c_{k_n + i}$. After simplifying all the possible $p_i$'s ($1 \le i \le \binom{k_n + 1}{2}$) in the denominators of the right-hand side of (\ref{eq:1}), it becomes a sum of elements of $A_n$ containing at least two summands, which contradicts the fact that $A_n$ generates $M_n$ minimally.
	
	Now define
	\[
		M = \bigcup_{n \in \nn} M_n.
	\]
	Because $\mathcal{A}(M_n) = A_n$ and $A_n \subset A_{n+1}$ for each $n \in \nn$, Lemma~\ref{lem:union of minimally generated PMs} implies that $M$ is an atomic Puiseux monoid with $\mathcal{A}(M) = \{a_n \mid n \in \nn\}$. Finally, we verify that $\{2\} \notin \mathcal{L}(M)$. It suffices to show that $|\mathsf{L}(x)| > 1$ for all $x \in M$ such that $2 \in \mathsf{L}(x)$. Take an element $x \in M$ such that $2 \in \mathsf{L}(x)$, and choose two indices $i,j \in \nn$ such that $x = a_i + a_j$. Let $m$ be the minimum natural number such that $a_i, a_j \in A_m$. By the way we constructed the sequence $\{A_n\}$, there exist $a \in A_{m+1}$ and an odd prime number $p$ such that
	\[
		a = \frac{a_i + a_j}{p}.
	\]
	Because $a \in \mathcal{A}(M)$, the element $x$ contains a factorization of length $p$, namely $pa$. Since $p > 2$, it follows that $|\mathsf{L}(x)| \ge |\{2,p\}| = \{2\}$, as desired.
\end{proof}

Recall that an atomic monoid $M$ is said to be completely non-half-factorial if for each $x \in M^\bullet \setminus \mathcal{A}(M)$ one has $|\mathsf{L}(x)| > 1$. It is not hard to argue that the monoid provided by Theorem~\ref{thm:PM with no set of lengths equal {2}} is completely non-half-factorial. Indeed, it satisfies condition (\ref{eq:a PM with extreme elementwise length}), which is stronger than completely non-half-factoriality.

\begin{cor} \label{cor:completely non-half-factorial PMs exist}
	There exists an atomic Puiseux monoid $M$ such that
	\begin{align} \label{eq:a PM with extreme elementwise length}
		\{|\mathsf{L}(x)| \ {\bf |} \ x \in M\} = \{1, \infty\}.
	\end{align}
\end{cor}

\begin{proof}
		Let $M$ be the Puiseux monoid constructed in Theorem~\ref{thm:PM with no set of lengths equal {2}}. Take $x \in M$ such that $|\mathsf{L}(x)| > 1$. This means that $x$ is neither zero nor an atom. Let $a_1$ and $a_2$ be two atoms of $M$ such that $y = a_1 + a_2$ divides $x$ in $M$. It follows by the construction of $M$ in the proof of Theorem~\ref{thm:PM with no set of lengths equal {2}} that $|\mathsf{L}(y)| = \infty$. The fact that $y$ divides $x$ in $M$ implies now that $|\mathsf{L}(x)| = \infty$, which concludes the proof.
\end{proof}

\begin{cor} \label{cor:intersection of nontrivial atomic PM}
	We have
	\[
		\bigcap \mathcal{L}(M) = \big\{\{0\}, \{1\} \big\},
	\]
	where the intersection is taking over all nontrivial atomic Puiseux monoids.
\end{cor}

\begin{proof}
	This is an immediate consequence of Theorem~\ref{thm:intersecion of SSL for NM} and Theorem~\ref{thm:PM with no set of lengths equal {2}}. \\
\end{proof}

\vspace{3pt}

\section{Relation with the Goldbach's Conjecture} \label{sec:Goldbach conjecture}

We conclude this paper providing evidence that explicit computations of sets of lengths is an extremely hard problem even in particular cases of atomic Puiseux monoids. Specifically, we shall prove that finding $\mathcal{L}(M)$ for $M = \langle p \mid p \ \text{ is prime } \rangle$ is as hard as the famous longstanding Goldbach's conjecture.

\begin{conj} [Goldbach's conjecture]
	Every even $n \ge 4$ can be expressed as the sum of two prime numbers.
\end{conj}

The following weaker version of the Goldbach's conjecture, called the Goldbach's weak conjecture, was proved in 2013 by Helfgott \cite{hH14}. \\

\begin{theorem} \label{thm:ternary Goldbach's conjecture}
	Every odd $n \ge 7$ can be written as the sum of three prime numbers.
\end{theorem}

We call the Puiseux monoid $M = \langle 1/p \mid p \text{ is prime } \rangle$ the \emph{elementary Puiseux monoid}. It was proved in \cite{GG17} that $M$ is \emph{hereditarily atomic} (i.e., every submonoid of $M$ is atomic). On the other hand, it follows immediately that $M$ is not a BF-monoid. For every $n \in \nn$ and $k = 1,\dots,n$, set
\[
	S_{n,k} := \{(a_1, \dots, a_k) \in \nn^k \mid a_1 + \dots + a_k = n\}.
\]
In the following proposition we characterize the sets of lengths of the elements of $M \cap \nn$ in terms of the $S_{n,k}$'s.

\begin{prop} \label{prop:length of the naturals inside the primary PM}
	Let $M$ be the elementary Puiseux monoid. Then for each $n \in M \cap \nn$, we have that
	\begin{align} \label{eq:elementary primary main equation}
		\mathsf{L}(n) = \bigcup_{k=1}^n \bigg\{\sum_{i=1}^k a_i p_i \ \bigg{|} \ (a_1, \dots, a_k) \in S_{n,k} \ \text{ and } \ p_1, \dots, p_k \ \text{are primes}\bigg\}.
	\end{align}
\end{prop}

\begin{proof}
	Take $n \in M \cap \nn$, and denote the right-hand side of (\ref{eq:elementary primary main equation}) by $R$. Note that assuming the extra condition $p_1 < \dots < p_k$ does not change $R$. Suppose that $\ell \in \mathsf{L}(n)$. Then there exist $k \in \nn$ and distinct prime numbers $p_1, \dots, p_k$ such that
	\begin{align} \label{eq:elementary primary eq1}
		n = c_1 \frac 1{p_1} + \dots + c_k \frac 1{p_k},
	\end{align}
	where the right-hand side of (\ref{eq:elementary primary eq1}) has length $\ell = c_1 + \dots + c_k$ when seen as a factorization of $n$. Applying the $p_i$-valuation map in both sides of (\ref{eq:elementary primary eq1}), we find that $p_i \mid c_i$ for $i = 1, \dots, k$. Now setting $a_i := c_i/p_i$ for $i = 1, \dots, k$, we get that $a_1 + \dots + a_k \in S_{n,k}$ and, therefore, $\ell = a_1 p_1 + \dots + a_k p_k \in R$. Thus, $\mathsf{L}(n) \subseteq R$. Conversely, if $(a_1, \dots, a_k) \in S_{n,k}$ and $p_1, \dots, p_k$ are distinct prime numbers, then it immediately follows that
	\[
		z \ : \ \sum_{i=1}^k (a_i p_i) \frac 1{p_i}
	\]
	is a factorization of $n = a_1 + \dots + a_k \in M \cap \nn$ with $|z| = a_1 p_1 + \dots + a_k p_k$. Therefore $R \subseteq \mathsf{L}(n)$, which completes the proof. 
\end{proof}

A natural number is said to be a \emph{Goldbach's number} if it can be expressed as the sum of two prime numbers (not necessarily distinct). Let $\mathsf{G}$ denote the set of Goldbach's numbers. If $M$ is the elementary primary Puiseux monoid, then an explicit description of $\mathsf{L}(2)$ is as hard as the Goldbach's conjecture.

\begin{theorem}
	Let $M$ be the elementary primary Puiseux monoid. Then
	\begin{enumerate}
		\item $\mathsf{L}(2) = \mathsf{G}$. \vspace{2pt}
		\item $\mathsf{L}(3) = \zz_{\ge 7}$
	\end{enumerate}
\end{theorem}

\begin{proof}
	First, we verify part (1). To see that $\mathsf{L}(2) \subseteq \mathsf{G}$, take $z \in \mathsf{Z}(2)$. If $z$ consists of copies of only one atom, say $1/p_i$, then $|z| = 2p_i \in \mathsf{G}$. Otherwise, by Proposition~\ref{prop:length of the naturals inside the primary PM}, $z$ is the formal sum of copies of two atoms, that is, $2 = a_1(1/p_1) + a_2(1/p_2)$ for distinct prime numbers $p_1$ and $p_2$ and positive coefficients $a_1$ and $a_2$. In this case, $p_1 \mid a_1$ and $p_2 \mid a_2$, which force $a_1 = p_1$ and $a_2 = p_2$. As a result, $|z| = a_1 + a_2 = p_1 + p_2 \in \mathsf{G}$. On the other hand, for each Goldbach number $p_1 + p_2$, the expression $p_1 (1/p_1) + p_2 (1/p_2)$ yields an element of $\mathsf{Z}(2)$ of length $p_1 + p_2$, which implies that $\mathsf{G} \subseteq \mathsf{L}(2)$. Thus, (1) follows. The proof of part (2) uses Theorem~\ref{thm:ternary Goldbach's conjecture}; however it follows similarly to the proof of part (1) and so it is left to the reader. \\
\end{proof}

\section{Acknowledgements}

	While working on this paper, the author was supported by the NSF-AGEP fellowship. The author would like to thank Alfred Geroldinger for helpful suggestions.

\end{document}